\documentclass{article}
\usepackage{amsmath,amsfonts,amssymb,amsthm,mathrsfs,bbm,stmaryrd,array}
\newtheorem{lemma}{Lemma}

\newtheorem{theorem}{Theorem}

\def\mydash{\CJKglue\raise0.2ex\hbox{---\kern-0.01em---}\CJKglue}

\begin{document}
\title{A congruence involving alternating harmonic sums modulo $p^{\alpha}q^{\beta}$\footnotetext{\noindent This work is  supported by the National Natural Science Foundation of China, Project (No.10871169) and the Natural Science Foundation of Zhejiang Province, Project (No. LQ13A010012).}}
\author{\scshape Zhongyan Shen$^{1\ast}$    Tianxin Cai$^{2\ast\ast}$ \\
1 Department of Mathematics, Zhejiang International Study University\\Hangzhou 310012, P.R. China\\
2 Department of Mathematics, Zhejiang University\\ Hangzhou 310027, P.R. China\\
$\ast$huanchenszyan@163.com\\
 $\ast\ast$txcai@zju.edu.cn}
\date{}
\maketitle
\textbf{Abstract} In 2014, Wang and Cai established the following
harmonic congruence for any odd prime $p$ and positive integer $r$,
\begin{equation*}
 \sum\limits_{i+j+k=p^{r}\atop{i,j,k\in \mathcal{P}_{p}}}\frac{1}{ijk}\equiv-2p^{r-1}B_{p-3} ~(\bmod ~ p^{r}),
\end{equation*}
where $\mathcal{P}_{n}$ denote the set of positive integers which
are prime to $n$.\\
In this note, we  obtain the congruences for distinct odd primes $p,~q$ and positive integers $\alpha,~\beta$,
\begin{equation*}
   \sum\limits_{\substack{i+j+k=p^{\alpha}q^{\beta}\\i,j,k\in\mathcal{P}_{pq}\\i\equiv j\equiv k\equiv 1\pmod{2}}}\frac{1}{ijk}\equiv\frac{7}{8}(2-q)(1-\frac{1}{q^{3}})p^{\alpha-1}q^{\beta-1}B_{p-3}\pmod{p^{\alpha}}
  \end{equation*}
and
 \begin{equation*}
 \sum\limits_{i+j+k=p^{\alpha}q^{\beta}\atop{i,j,k\in \mathcal{P}_{pq}}}\frac{(-1)^{i}}{ijk}
 \equiv
 \frac{1}{2}(q-2)(1-\frac{1}{q^{3}})p^{\alpha-1}q^{\beta-1}B_{p-3}\pmod{p^{\alpha}}.
\end{equation*}
Finally, we raise a conjecture that for $n>1$ and odd prime power
$p^{\alpha}||n$, $\alpha\geq1$,
\begin{eqnarray}
    \nonumber \sum\limits_{i+j+k=n\atop{i,j,k\in\mathcal{P}_{n}}}\frac{(-1)^{i}}{ijk}
    \equiv \prod\limits_{q|n\atop{q\neq p}}(1-\frac{2}{q})(1-\frac{1}{q^{3}})\frac{n}{2p}B_{p-3}\pmod{p^{\alpha}}
  \end{eqnarray}
  and
\begin{eqnarray}
    \nonumber  \sum\limits_{\substack{i+j+k=n\\i,j,k\in\mathcal{P}_{n}\\i\equiv j\equiv k\equiv 1\pmod{2}}}\frac{1}{ijk}
    \equiv \prod\limits_{q|n\atop{q\neq p}}(1-\frac{2}{q})(1-\frac{1}{q^{3}})(-\frac{7n}{8p})B_{p-3}\pmod{p^{\alpha}}.
  \end{eqnarray}


\section{Introduction.}
Let
\begin{equation*}
 Z(n)=\sum\limits_{i+j+k=n\atop{i,j,k\in\mathcal{P}_{n}}}\frac{1}{ijk},
\end{equation*}
where $\mathcal{P}_{n}$ denote the set of positive integers which
are prime to $n$.\\

At the beginning of the 21th century, Zhao (Cf.\cite{Zh}) first
announced the following curious congruence involving multiple
harmonic sums for any odd prime $p>3$,
\begin{equation}
 Z(p)\equiv-2B_{p-3} ~(\bmod ~ p), \label{eq:1}
\end{equation}
which holds when $p=3$ evidently. Here, Bernoulli numbers $B_{k}$
are defined by the recursive relation:
\begin{center}
 $\sum\limits_{i=0}^{n}\binom{n+1}{i}B_{i}=0,n\geq 1 .$
\end{center}
A simple proof of  \eqref{eq:1} was presented in \cite{Ji}.
Later, Xia and Cai \cite{XC} generalized \eqref{eq:1} to
\begin{equation*}
Z(p)\equiv-\frac{12B_{p-3}}{p-3}-\frac{3B_{2p-4}}{p-4} ~(\bmod ~ p^{2}),
\end{equation*}
where $p>5$ is a prime. \\
In 2014, Wang and Cai \cite{WC} proved for every prime $p\geq 3$
and positive integer $r$,
\begin{equation}
Z(p^{r})\equiv-2p^{r-1}B_{p-3} ~(\bmod ~ p^{r}).\label{eq:2}
\end{equation}
Let $n=2$ or 4, for every positive integer $r\geq \frac{n}{2}$ and
prime $p> n$, Zhao \cite{Zh1} generalized \eqref{eq:2} to
\begin{equation*}
 \sum\limits_{i_{1}+i_{2}+\cdots+i_{n}=p^{r}\atop{i_{1},i_{2},\cdots,i_{n}\in \mathcal{P}_{p}}}\frac{1}{i_{1}i_{2}\cdots i_{n}}
 \equiv-\frac{n!}{n+1}p^{r}B_{p-n-1} ~(\bmod  ~p^{r+1}).
\end{equation*}
Recently, for distinct odd primes $p,~q$ and positive integers
$\alpha,~\beta$, the authors and Jia\cite{CSJ} proved that
 \begin{equation}
  Z(p^{\alpha}q^{\beta})\equiv 2(2-q)(1-\frac{1}{q^{3}})p^{\alpha-1}q^{\beta-1}B_{p-3}\pmod{p^{\alpha}},\label{eq:10}
  \end{equation}
and conjecture that for $n>1$, $p^{\alpha}||n$, $\alpha\geq1$,
\begin{eqnarray}
    \nonumber Z(n)\equiv \prod\limits_{q|n\atop{q\neq p}}(1-\frac{2}{q})(1-\frac{1}{q^{3}})(-\frac{2n}{p})B_{p-3}\pmod{p^{\alpha}}.
  \end{eqnarray}
 This is the generalization of  \eqref{eq:2} and \eqref{eq:10}.\\

In this paper, we consider the congruences involving  alternating harmonic sums and obtain the following theorems.

\begin{theorem}
Let $p,~q$ be distinct odd primes and $\alpha,~\beta$ positive integer, then
 \begin{equation*}
 \sum\limits_{i+j+k=2p^{\alpha}q^{\beta}\atop{i,j,k\in\mathcal{P}_{pq}}}\frac{(-1)^{i}}{ijk}
  =-\frac{1}{2}Z(p^{\alpha}q^{\beta})\equiv (q-2)(1-\frac{1}{q^{3}})p^{\alpha-1}q^{\beta-1}B_{p-3}\pmod{p^{\alpha}}.
  \end{equation*}
\end{theorem}

\begin{theorem}
Let $p,~q$ be distinct odd primes and $\alpha,~\beta$ positive integer, then
 \begin{equation*}
 \sum\limits_{i+j+k=p^{\alpha}q^{\beta}\atop{i,j,k\in\mathcal{P}_{pq}}}\frac{(-1)^{i}}{ijk}
  \equiv -\frac{1}{4}Z(p^{\alpha}q^{\beta})\equiv\frac{1}{2}(q-2)(1-\frac{1}{q^{3}})p^{\alpha-1}q^{\beta-1}B_{p-3}\pmod{p^{\alpha}}.
  \end{equation*}
\end{theorem}

\begin{theorem}
Let $p,~q$ be distinct odd primes and $\alpha,~\beta$ positive integer, then
 \begin{equation*}
   \sum\limits_{\substack{i+j+k=p^{\alpha}q^{\beta}\\i,j,k\in\mathcal{P}_{pq}\\i\equiv j\equiv k\equiv 1\pmod{2}}}\frac{1}{ijk}\equiv\frac{7}{16}Z(p^{\alpha}q^{\beta})
  \equiv\frac{7}{8}(2-q)(1-\frac{1}{q^{3}})p^{\alpha-1}q^{\beta-1}B_{p-3}\pmod{p^{\alpha}}.
  \end{equation*}
\end{theorem}
Finally, we have the following\\

\textbf{Conjecture}~ For any positive integer $n>1$, if odd prime
power $p^{\alpha}||n(p^{\alpha}|n,~p^{\alpha+1}\nmid n)$,
$\alpha\geq 1$, then
\begin{eqnarray}
    \nonumber \sum\limits_{i+j+k=n\atop{i,j,k\in\mathcal{P}_{n}}}\frac{(-1)^{i}}{ijk}
    \equiv \prod\limits_{q|n\atop{q\neq p}}(1-\frac{2}{q})(1-\frac{1}{q^{3}})\frac{n}{2p}B_{p-3}\pmod{p^{\alpha}}
  \end{eqnarray}
  and
\begin{eqnarray}
    \nonumber  \sum\limits_{\substack{i+j+k=n\\i,j,k\in\mathcal{P}_{n}\\i\equiv j\equiv k\equiv 1\pmod{2}}}\frac{1}{ijk}
    \equiv \prod\limits_{q|n\atop{q\neq p}}(1-\frac{2}{q})(1-\frac{1}{q^{3}})(-\frac{7n}{8p})B_{p-3}\pmod{p^{\alpha}}.
  \end{eqnarray}
  This is the generalization of  Theorem 2 and Theorem 3.\\
\section{Preliminaries.}
In order to prove the theorems, we need the following lemmas.
\begin{lemma}
Let $p,~q$ be distinct odd primes, $m$ positive integer coprime to
$pq$ and $\alpha,~\beta$ positive integers, then
 \begin{equation*}
 \sum\limits_{i+j+k=mp^{\alpha}q^{\beta}\atop{i,j,k\in\mathcal{P}_{pq}}}\frac{1}{ijk}\equiv mZ(p^{\alpha}q^{\beta})~\pmod{p^{\alpha}}.
\end{equation*}
\end{lemma}
\begin{proof}
For every triple $(i,~j,~k)$ of positive integers which satisfies
$i+j+k= mp^{\alpha}q^{\beta}$, $i,j,k\in\mathcal{P}_{pq}$, we
rewrite
 \begin{equation*}
i=xp^{\alpha}q^{\beta}+i_{0}, ~j = yp^{\alpha}q^{\beta} + j_{0},~
k = zp^{\alpha}q^{\beta}+ k_{0},
  \end{equation*}
where $1\leq i_{0},~j_{0},~k_{0}< p^{\alpha}q^{\beta}$ are prime
to
$pq$ and integers $x,~y,~z\geq 0$.\\
Since $3\leq i_{0}+j_{0}+k_{0}< 3p^{\alpha}q^{\beta}$ and
$i_{0}+j_{0}+k_{0}=(m-x-y-z)p^{\alpha}q^{\beta}$, it is easy to
see that
\begin{align}
&\begin{cases}
 i_{0}+j_{0}+k_{0}=p^{\alpha}q^{\beta}\\
x+y+z=m-1\\
  \end{cases}
 ~~or~~ \begin{cases}
i_{0}+j_{0}+k_{0}=2p^{\alpha}q^{\beta}\\
x+y+z=m-2\\
  \end{cases}.\nonumber
  \end{align}
  Hence
\begin{align}
&
\sum\limits_{i+j+k=mp^{\alpha}q^{\beta}\atop{i,j,k\in\mathcal{P}_{pq}}}\frac{1}{ijk}
=\sum\limits_{i_{0}+j_{0}+k_{0}=p^{\alpha}q^{\beta}\atop{x+y+z=m-1,i_{0},j_{0},k_{0}\in\mathcal{P}_{pq}}}
 \frac{1}{(xp^{\alpha}q^{\beta}+i_{0})(yp^{\alpha}q^{\beta}+j_{0})(zp^{\alpha}q^{\beta}+k_{0})}
\nonumber\\&+\sum\limits_{i_{0}+j_{0}+k_{0}=2p^{\alpha}q^{\beta},1\leq
i_{0},j_{0},k_{0}<p^{\alpha}q^{\beta}\atop{x+y+z=m-2,i_{0},j_{0},k_{0}\in\mathcal{P}_{pq}}}
\frac{1}{(xp^{\alpha}q^{\beta}+i_{0})(yp^{\alpha}q^{\beta}+j_{0})(zp^{\alpha}q^{\beta}+k_{0})}
\nonumber\\&\equiv\sum\limits_{i_{0}+j_{0}+k_{0}=p^{\alpha}q^{\beta}\atop{x+y+z=m-1,i_{0},j_{0},k_{0}\in\mathcal{P}_{pq}}}
 \frac{1}{i_{0}j_{0}k_{0}}
+\sum\limits_{i_{0}+j_{0}+k_{0}=2p^{\alpha}q^{\beta},1\leq
i_{0},j_{0},k_{0}<p^{\alpha}q^{\beta}\atop{x+y+z=m-2,i_{0},j_{0},k_{0}\in\mathcal{P}_{pq}}}
\frac{1}{i_{0}j_{0}k_{0}}
\nonumber\\&\equiv\binom{m+1}{2}\sum\limits_{i_{0}+j_{0}+k_{0}=p^{\alpha}q^{\beta}\atop{i_{0},j_{0},k_{0}\in\mathcal{P}_{pq}}}
 \frac{1}{i_{0}j_{0}k_{0}}
+\binom{m}{2}\sum\limits_{\substack{i_{0}+j_{0}+k_{0}=2p^{\alpha}q^{\beta}\\1\leq
i_{0},j_{0},k_{0}<p^{\alpha}q^{\beta}\\
i_{0},j_{0},k_{0}\in\mathcal{P}_{pq}}}
\frac{1}{i_{0}j_{0}k_{0}},\label{11}
\end{align}
here in the last equation we use the fact that there are
$\binom{m+1}{2}$ or $\binom{m}{2}$ triples $(x,~y,~z)$ of
 nonnegative integers which satisfy $x+y+z=m-1$ or $x+y+z=m-2$,
 respectively. For the second sum in \eqref{11}, note that
 $(i_{0},j_{0},k_{0})\leftrightarrow
 (p^{\alpha}q^{\beta}-i_{0},p^{\alpha}q^{\beta}-j_{0},p^{\alpha}q^{\beta}-k_{0})$
  gives a bijection between the solutions of
  $i_{0}+j_{0}+k_{0}=p^{\alpha}q^{\beta}$ and
  $i_{0}+j_{0}+k_{0}=2p^{\alpha}q^{\beta}$, where $1\leq
i_{0},j_{0},k_{0}<p^{\alpha}q^{\beta}$, thus, we have
\begin{align}
\sum\limits_{\substack{i_{0}+j_{0}+k_{0}=2p^{\alpha}q^{\beta}\\1\leq
i_{0},j_{0},k_{0}<p^{\alpha}q^{\beta}\\
i_{0},j_{0},k_{0}\in\mathcal{P}_{pq}}} \frac{1}{i_{0}j_{0}k_{0}}
&=\sum\limits_{i_{0}+j_{0}+k_{0}=p^{\alpha}q^{\beta}\atop{i_{0},j_{0},k_{0}\in\mathcal{P}_{pq}}}
\frac{1}{(p^{\alpha}q^{\beta}-i_{0}((p^{\alpha}q^{\beta}-j_{0})(p^{\alpha}q^{\beta}-k_{0})}
\nonumber\\&\equiv-\sum\limits_{i_{0}+j_{0}+k_{0}=p^{\alpha}q^{\beta}\atop{i_{0},j_{0},k_{0}\in\mathcal{P}_{pq}}}
\frac{1}{i_{0}j_{0}k_{0}}\pmod{p^{\alpha}q^{\beta}}.\nonumber
\end{align}
Hence, \eqref{11} is congruent to
\begin{align}
\binom{m+1}{2}\sum\limits_{i_{0}+j_{0}+k_{0}=p^{\alpha}q^{\beta}\atop{i_{0},j_{0},k_{0}\in\mathcal{P}_{pq}}}
 \frac{1}{i_{0}j_{0}k_{0}}
-\binom{m}{2}\sum\limits_{i_{0}+j_{0}+k_{0}=p^{\alpha}q^{\beta}\atop{i_{0},j_{0},k_{0}\in\mathcal{P}_{pq}}}
 \frac{1}{i_{0}j_{0}k_{0}}
 \equiv mZ(p^{\alpha}q^{\beta})\pmod{p^{\alpha}q^{\beta}}.
\end{align}
 Then we complete the proof of Lemma 1.
\end{proof}
\begin{lemma}[\cite{CSJ}]
Let $p,~q$ be distinct odd primes and $\alpha,~\beta$ positive integers, if and only if $p=q^{2}+q+1$ or $q=p^{2}+p+1$ or $p|q^{2}+q+1$ and $q|p^{2}+p+1$, we have
 \begin{equation*}
  Z(p^{\alpha}q^{\beta})\equiv 0\pmod{p^{\alpha}q^{\beta}}.
  \end{equation*}
  \end{lemma}

\section{Proofs of the theorems.}
\begin{proof}[Proof of Theorem 1 \\]
Since $i+j+k=2p^{\alpha}q^{\beta}$ is even, either $i,~j,~k$ are
all even or one of $i,~j,~k$ is even and the other two are odd,
then
\begin{eqnarray}\label{3}
      \sum\limits_{i+j+k=2p^{\alpha}q^{\beta}\atop{i,j,k\in\mathcal{P}_{pq}}}\frac{1}{ijk}
   =\sum\limits_{i+j+k=2p^{\alpha}q^{\beta},i,j,k\in\mathcal{P}_{pq}\atop{i,j,k~are~all~even}}\frac{1}{ijk}
   +\sum\limits_{i+j+k=2p^{\alpha}q^{\beta},i,j,k\in\mathcal{P}_{pq}\atop{exactly ~one~ of~ i,j,k~is~even}}\frac{1}{ijk}.
  \end{eqnarray}
By  symmetry, we have
 \begin{eqnarray}\label{4}
      \sum\limits_{i+j+k=2p^{\alpha}q^{\beta}\atop{i,j,k\in\mathcal{P}_{pq}}}\frac{(-1)^{i}}{ijk}
   =\frac{1}{3} \sum\limits_{i+j+k=2p^{\alpha}q^{\beta}\atop{i,j,k\in\mathcal{P}_{pq}}}\frac{(-1)^{i}+(-1)^{j}+(-1)^{k}}{ijk}.
  \end{eqnarray}
  If $i,~j,~k$ are all even, the right hand of (\ref{4})  equals to
  \begin{eqnarray}\label{5}
  \sum\limits_{i+j+k=2p^{\alpha}q^{\beta},i,j,k\in\mathcal{P}_{pq}\atop{i,j,k~are~all~even}}\frac{1}{ijk}
  =\frac{1}{8}\sum\limits_{i+j+k=p^{\alpha}q^{\beta}\atop{i,j,k\in\mathcal{P}_{pq}}}\frac{1}{ijk},
     \end{eqnarray}
     where we replace $i,~j,~k$ by $2i,~2j,~2k$ respectively.\\

 If one of $i,~j,~k$ is even and the other two are odd, the right hand of (\ref{4})  equals to
 \begin{eqnarray}
  \nonumber   -\frac{1}{3}\sum\limits_{i+j+k=2p^{\alpha}q^{\beta},i,j,k\in\mathcal{P}_{pq}\atop{exactly ~one~ of~ i,j,k~is~even}}\frac{1}{ijk}.
  \end{eqnarray}
 Thus, (\ref{4})  is equal to
 \begin{eqnarray}
  \nonumber  \sum\limits_{i+j+k=2p^{\alpha}q^{\beta}\atop{i,j,k\in\mathcal{P}_{pq}}}\frac{(-1)^{i}}{ijk}
   &=&\sum\limits_{i+j+k=2p^{\alpha}q^{\beta},i,j,k\in\mathcal{P}_{pq}\atop{i,j,k~are~all~even}}\frac{1}{ijk} -\frac{1}{3}\sum\limits_{i+j+k=2p^{\alpha}q^{\beta},i,j,k\in\mathcal{P}_{pq}\atop{exactly ~one~ of~ i,j,k~is~even}}\frac{1}{ijk}\\
  \nonumber &=&\frac{4}{3}\sum\limits_{i+j+k=2p^{\alpha}q^{\beta},i,j,k\in\mathcal{P}_{pq}\atop{i,j,k~are~all~even}}\frac{1}{ijk} -\frac{1}{3}(\sum\limits_{i+j+k=2p^{\alpha}q^{\beta},i,j,k\in\mathcal{P}_{pq}\atop{i,j,k~are~all~even}}\frac{1}{ijk}\\
  \nonumber & &+\sum\limits_{i+j+k=2p^{\alpha}q^{\beta},i,j,k\in\mathcal{P}_{pq}\atop{exactly ~one~ of~ i,j,k~is~even}}\frac{1}{ijk}).
  \end{eqnarray}
  By using (\ref{3}) and (\ref{5}), we have
  \begin{eqnarray}
  \nonumber  \sum\limits_{i+j+k=2p^{\alpha}q^{\beta}\atop{i,j,k\in\mathcal{P}_{pq}}}\frac{(-1)^{i}}{ijk}
  =\frac{1}{6}Z(p^{\alpha}q^{\beta}) -\frac{1}{3}\sum\limits_{i+j+k=2p^{\alpha}q^{\beta}\atop{i,j,k\in\mathcal{P}_{pq}}}\frac{1}{ijk}.
  \end{eqnarray}
  It follows from Lemma 1 that
  \begin{eqnarray}
  \nonumber  \sum\limits_{i+j+k=2p^{\alpha}q^{\beta}\atop{i,j,k\in\mathcal{P}_{pq}}}\frac{(-1)^{i}}{ijk}
  =\frac{1}{6}Z(p^{\alpha}q^{\beta}) -\frac{2}{3}Z(p^{\alpha}q^{\beta})=-\frac{1}{2}Z(p^{\alpha}q^{\beta})\pmod{p^{\alpha}}.
  \end{eqnarray}
 By using \eqref{eq:10}, we complete the proof of Theorem 1.
\end{proof}
\begin{proof}[Proof of Theorem 2 \\]
For every triple $(i,~j,~k)$ of positive integers which satisfies
$i+j+k=2p^{\alpha}q^{\beta},~i,j,k\in \mathcal{P}_{pq}$, we take it into 3 cases.\\

Cases 1. If $1\leq i,~j,~k\leq p^{\alpha}q^{\beta}-1$ are coprime
to $pq$, $(i,~j,~k)\leftrightarrow
(p^{\alpha}q^{\beta}-i,~p^{\alpha}q^{\beta}-j,~p^{\alpha}q^{\beta}-k)$
is a bijection between the solutions of
$i+j+k=2p^{\alpha}q^{\beta}$ and
$i+j+k=p^{\alpha}q^{\beta},~i,j,k\in \mathcal{P}_{pq}$, we have
 \begin{align}
 \sum_{\substack{i+j+k=2p^{\alpha}q^{\beta}\\i,j,k\in \mathcal{P}_{pq}\\1\leq i,~j,~k\leq p^{\alpha}q^{\beta}-1}}\frac{(-1)^{i}}{ijk}
 &\equiv \sum\limits_{i+j+k=p^{\alpha}q^{\beta}\atop{i,j,k\in \mathcal{P}_{pq}}}\frac{(-1)^{p^{\alpha}q^{\beta}-i}}{(p^{\alpha}q^{\beta}-i)(p^{\alpha}q^{\beta}-j)(p^{\alpha}q^{\beta}-k)}
\nonumber\\& \equiv \sum\limits_{i+j+k=p^{\alpha}q^{\beta}\atop{i,j,k\in
\mathcal{P}_{pq}}}\frac{(-1)^{i}}{ijk}\pmod{p^{\alpha}}.\label{eq:6}
\end{align}
Cases 2. If $p^{\alpha}q^{\beta}+1\leq i\leq
2p^{\alpha}q^{\beta}-1,1\leq j,~k\leq p^{\alpha}q^{\beta}-1$ are
coprime to $pq$, $(i,~j,~k)\leftrightarrow
(p^{\alpha}q^{\beta}+i,~j,~k)$ is a bijection between the
solutions of $i+j+k=2p^{\alpha}q^{\beta}$ and
$i+j+k=p^{\alpha}q^{\beta},~i,j,k\in \mathcal{P}_{pq}$, we have
 \begin{align}
 \sum_{\substack{i+j+k=2p^{\alpha}q^{\beta}\\i,j,k\in \mathcal{P}_{pq}\\p^{\alpha}q^{\beta}+1\leq i\leq 2p^{\alpha}q^{\beta}-1,1\leq j,~k\leq p^{\alpha}q^{\beta}-1}}\frac{(-1)^{i}}{ijk}
 &\equiv \sum\limits_{i+j+k=p^{\alpha}q^{\beta}\atop{i,j,k\in \mathcal{P}_{pq}}}\frac{(-1)^{p^{\alpha}q^{\beta}+i}}{(p^{\alpha}q^{\beta}+i)jk}
\nonumber\\& \equiv -\sum\limits_{i+j+k=p^{\alpha}q^{\beta}\atop{i,j,k\in
\mathcal{P}_{pq}}}\frac{(-1)^{i}}{ijk}\pmod{p^{\alpha}}.\label{eq:7}
\end{align}
Cases 3. If $p^{\alpha}q^{\beta}+1\leq j\leq
2p^{\alpha}q^{\beta}-1,1\leq i,~k\leq p^{\alpha}q^{\beta}-1$ or
$p^{\alpha}q^{\beta}+1\leq k\leq 2p^{\alpha}q^{\beta}-1,1\leq
i,~j\leq p^{\alpha}q^{\beta}-1$ are corime to $pq$,
$(i,~j,~k)\leftrightarrow (i,~p^{\alpha}q^{\beta}+j,~k)$ in the
former and $(i,~j,~k)\leftrightarrow
(i,~j,~p^{\alpha}q^{\beta}+k)$ in the later are the bijections
between the solutions of $i+j+k=2p^{\alpha}q^{\beta}$ and
$i+j+k=p^{\alpha}q^{\beta},~i,j,k\in \mathcal{P}_{pq}$, we have
 \begin{align}
 &\sum_{\substack{i+j+k=2p^{\alpha}q^{\beta}\\i,j,k\in \mathcal{P}_{pq}\\p^{r}+1\leq j\leq 2p^{\alpha}q^{\beta}-1,1\leq i,~k\leq p^{\alpha}q^{\beta}-1}}\frac{(-1)^{i}}{ijk}
+\sum_{\substack{i+j+k=2p^{\alpha}q^{\beta}\\i,j,k\in \mathcal{P}_{pq}\\p^{\alpha}q^{\beta}+1\leq
k\leq 2p^{\alpha}q^{\beta}-1,1\leq i,~j\leq p^{\alpha}q^{\beta}-1}}\frac{(-1)^{i}}{ijk}
 \nonumber\\&\equiv \sum\limits_{i+j+k=p^{\alpha}q^{\beta}\atop{i,j,k\in
 \mathcal{P}_{pq}}}\frac{(-1)^{i}}{i(p^{\alpha}q^{\beta}+j)k}+\sum\limits_{i+j+k=p^{\alpha}q^{\beta}\atop{i,j,k\in
 \mathcal{P}_{pq}}}\frac{(-1)^{i}}{ij(p^{\alpha}q^{\beta}+k)}
\nonumber\\& \equiv 2\sum\limits_{i+j+k=p^{\alpha}q^{\beta}\atop{i,j,k\in
\mathcal{P}_{pq}}}\frac{(-1)^{i}}{ijk}\pmod{p^{\alpha}}.\label{eq:8}
\end{align}
Combining \eqref{eq:6}-\eqref{eq:8}, it follows that
 \begin{align}
 &\sum\limits_{i+j+k=2p^{\alpha}q^{\beta}\atop{i,j,k\in \mathcal{P}_{pq}}}\frac{(-1)^{i}}{ijk}=
 \sum_{\substack{i+j+k=2p^{\alpha}q^{\beta}\\i,j,k\in \mathcal{P}_{pq}\\1\leq i,~j,~k\leq p^{\alpha}q^{\beta}-1}}\frac{(-1)^{i}}{ijk}
+ \sum_{\substack{i+j+k=2p^{\alpha}q^{\beta}\\i,j,k\in
\mathcal{P}_{pq}\\p^{\alpha}q^{\beta}+1\leq i\leq 2p^{\alpha}q^{\beta}-1,1\leq j,~k\leq
p^{\alpha}q^{\beta}-1}}\frac{(-1)^{i}}{ijk}
\nonumber\\&+\sum_{\substack{i+j+k=2p^{\alpha}q^{\beta}\\i,j,k\in
\mathcal{P}_{pq}\\p^{\alpha}q^{\beta}+1\leq j\leq 2p^{\alpha}q^{\beta}-1,1\leq i,~k\leq
p^{\alpha}q^{\beta}-1}}\frac{(-1)^{i}}{ijk}
+\sum_{\substack{i+j+k=2p^{\alpha}q^{\beta}\\i,j,k\in
\mathcal{P}_{pq}\\p^{\alpha}q^{\beta}+1\leq k\leq 2p^{\alpha}q^{\beta}-1,1\leq i,~j\leq
p^{\alpha}q^{\beta}-1}}\frac{(-1)^{i}}{ijk}
 \nonumber\\&\equiv 2\sum\limits_{i+j+k=p^{\alpha}q^{\beta}\atop{i,j,k\in
\mathcal{P}_{pq}}}\frac{(-1)^{i}}{ijk}\pmod{p^{\alpha}}.\nonumber
\end{align}
By Theorem 1, we complete the proof of Theorem 2.
\end{proof}

\begin{proof}[Proof of Theorem 3 \\]
 Since $i+j+k=p^{\alpha}q^{\beta}$ is odd, either $i,~j,~k$ are all odd or one of $i,~j,~k$ is odd and the other two are even, then
\begin{align}
     Z(p^{\alpha}q^{\beta})
   =\sum_{\substack{i+j+k=p^{\alpha}q^{\beta}\\i,j,k\in\mathcal{P}_{pq}\\i,j,k~are~all~odd}}\frac{1}{ijk}
   +\sum_{\substack{i+j+k=p^{\alpha}q^{\beta}\\i,j,k\in\mathcal{P}_{pq}\\exactly ~one~ of~ i,j,k~is~odd}}\frac{1}{ijk}.\label{9}
  \end{align}
By  symmetry, similar to Theorem 1, we have
 \begin{eqnarray}\label{10}
  \nonumber  \sum\limits_{i+j+k=p^{\alpha}q^{\beta}\atop{i,j,k\in\mathcal{P}_{pq}}}\frac{(-1)^{i}}{ijk}
  &=&\frac{1}{3} \sum\limits_{i+j+k=p^{\alpha}q^{\beta}\atop{i,j,k\in\mathcal{P}_{pq}}}\frac{(-1)^{i}+(-1)^{j}+(-1)^{k}}{ijk}\\
   \nonumber &=&-\sum\limits_{\substack{i+j+k=p^{\alpha}q^{\beta}\\i,j,k\in\mathcal{P}_{pq}\\i,j,k~are~all~odd}}\frac{1}{ijk} +\frac{1}{3}\sum\limits_{\substack{i+j+k=p^{\alpha}q^{\beta}\\i,j,k\in\mathcal{P}_{pq}\\exactly ~one~ of~ i,j,k~is~odd}}\frac{1}{ijk}\\
  \nonumber &=&-\frac{4}{3}\sum\limits_{\substack{i+j+k=p^{\alpha}q^{\beta}\\i,j,k\in\mathcal{P}_{pq}\\i,j,k~are~all~odd}}\frac{1}{ijk} +\frac{1}{3}(\sum\limits_{\substack{i+j+k=p^{\alpha}q^{\beta}\\i,j,k\in\mathcal{P}_{pq}\\i,j,k~are~all~odd}}\frac{1}{ijk}\\
  \nonumber & &+\sum\limits_{\substack{i+j+k=p^{\alpha}q^{\beta}\\i,j,k\in\mathcal{P}_{pq}\\exactly ~one~ of~ i,j,k~is~odd}}\frac{1}{ijk})\\.
   \nonumber &=&-\frac{4}{3}\sum\limits_{\substack{i+j+k=p^{\alpha}q^{\beta}\\i,j,k\in\mathcal{P}_{pq}\\i,j,k~are~all~odd}}\frac{1}{ijk}+\frac{1}{3}Z(p^{\alpha}q^{\beta}),
  \end{eqnarray}
 where we use (\ref{9}) in the last equation.

By Theorem 2, we have
 \begin{eqnarray}
  \nonumber \sum\limits_{\substack{i+j+k=p^{\alpha}q^{\beta}\\i,j,k\in\mathcal{P}_{pq}\\i,j,k~are~all~odd}}\frac{1}{ijk}\equiv\frac{7}{16}Z(p^{\alpha}q^{\beta})\pmod{p^{\alpha}}.
  \end{eqnarray}
 Therefore, we complete the proof of Theorem 3.
\end{proof}

\textbf{Remark 1}~ By Lemma 2 and Theorem 3 in \cite{CSJ}, let
$p,~q$ be distinct odd primes and $\alpha,~\beta$ positive
integers, if and only if $p=q^{2}+q+1$ or $q=p^{2}+p+1$ or
$p|q^{2}+q+1$ and $q|p^{2}+p+1$, we have
 \begin{equation*}
 \sum\limits_{i+j+k=p^{\alpha}q^{\beta}\atop{i,j,k\in\mathcal{P}_{pq}}}\frac{(-1)^{i}}{ijk}\equiv 0\pmod{p^{\alpha}q^{\beta}}
  \end{equation*}
and
\begin{equation*}
   \sum\limits_{\substack{i+j+k=p^{\alpha}q^{\beta}\\i,j,k\in\mathcal{P}_{pq}\\i\equiv j\equiv k\equiv 1\pmod{2}}}\frac{1}{ijk}\equiv 0\pmod{p^{\alpha}q^{\beta}}.
  \end{equation*}
  When $n=pq$, $p,~q$ are distinct odd primes, in \cite{CSJ},we have
  \begin{eqnarray}
    \nonumber Z(n)\equiv 6(1+\frac{3}{\phi(n)-2})(1+\frac{1}{(\phi(n)-1)^{3}})B_{\phi(n)-2}\pmod{n}.
  \end{eqnarray}
  Hence, by Theorem 2 and Theorem 3,  we have

  \begin{equation*}
 \sum\limits_{i+j+k=n\atop{i,j,k\in\mathcal{P}_{n}}}\frac{(-1)^{i}}{ijk}\equiv -\frac{3}{2}(1+\frac{3}{\phi(n)-2})(1+\frac{1}{(\phi(n)-1)^{3}})B_{\phi(n)-2}\pmod{n}
  \end{equation*}
  and
\begin{equation*}
   \sum\limits_{\substack{i+j+k=n\\i,j,k\in\mathcal{P}_{n}\\i\equiv j\equiv k\equiv 1\pmod{2}}}\frac{1}{ijk}\equiv \frac{21}{8}(1+\frac{3}{\phi(n)-2})(1+\frac{1}{(\phi(n)-1)^{3}})B_{\phi(n)-2}\pmod{n}.
  \end{equation*}
  In particular, for any odd prime $p$ and positive integer $r$, we have
  \begin{equation*}
 \sum\limits_{i+j+k=p^{r}\atop{i,j,k\in\mathcal{P}_{p}}}\frac{(-1)^{i}}{ijk}\equiv \frac{1}{2}p^{r-1}B_{p-3} ~(\bmod ~ p^{r})
  \end{equation*}
  and
\begin{equation*}
   \sum\limits_{\substack{i+j+k=p^{r}\\i,j,k\in\mathcal{P}_{p}\\i\equiv j\equiv k\equiv 1\pmod{2}}}\frac{1}{ijk}\equiv -\frac{7}{8}p^{r-1}B_{p-3} ~(\bmod ~ p^{r}).
  \end{equation*}

 \textbf{Remark 2}~By the conjecture and Chinese Remainder Theorem, we can count out
  the remainder of
\begin{eqnarray}
    \nonumber
    \sum\limits_{i+j+k=n\atop{i,j,k\in\mathcal{P}_{n}}}\frac{(-1)^{i}}{ijk}~~~~and~~~~
    \sum\limits_{\substack{i+j+k=n\\i,j,k\in\mathcal{P}_{n}\\i\equiv j\equiv k\equiv 1\pmod{2}}}\frac{1}{ijk}
  \end{eqnarray}
   modulo $n$ for any positive integer
  $n$.\\

  \textbf{Problem 1}~Can we find arithmetical functions $f(n)$ and $g(n)$ such that
\begin{equation*}
  \sum\limits_{i+j+k=n\atop{i,j,k\in\mathcal{P}_{n}}}\frac{(-1)^{i}}{ijk}\equiv f(n)\pmod{n}~~~~and~~~~
 \sum\limits_{\substack{i+j+k=n\\i,j,k\in\mathcal{P}_{n}\\i\equiv j\equiv k\equiv 1\pmod{2}}}\frac{1}{ijk} \equiv g(n)\pmod{n}.
\end{equation*}




\end{document}